\documentclass[oneside,12pt]{article}

\usepackage[utf8]{inputenc}
\usepackage{amsmath}
\usepackage{amsthm}
\usepackage{amssymb}
\usepackage{times}
\usepackage{bm}
\usepackage[numbers]{natbib}
\usepackage{amsfonts}
\usepackage{amssymb}
\usepackage{authblk}
\usepackage{mathtools}
\usepackage{dsfont}
\usepackage[margin=1in]{geometry}
\usepackage{xcolor}
\usepackage{array}
\usepackage{booktabs}
\usepackage{caption}
\usepackage{tikz}
\usepackage[inline]{enumitem}

\usetikzlibrary{positioning}
\usepackage{graphicx}
\RequirePackage[colorlinks,citecolor=blue,urlcolor=blue]{hyperref}
\usepackage[norefs,nocites]{refcheck}

\DeclarePairedDelimiterX{\norm}[1]{\lVert}{\rVert}{#1}
\DeclarePairedDelimiterX{\abs}[1]{\lvert}{\rvert}{#1}

\DeclareMathOperator*{\argmax}{arg\,max}

\newcommand\restr[2]{{
  \left.\kern-\nulldelimiterspace 
  #1 
  \right|_{#2} 
  }}

\newcommand{\bbP}{\mathbb{P}}

\newcommand{\bbR}{\mathbb{R}}

\newcommand{\bbE}{\mathbb{E}}
\newcommand{\bbN}{\mathbb{N}}

\newcommand{\calA}{\mathcal{A}}
\newcommand{\calH}{\mathcal{H}}
\newcommand{\calD}{\mathcal{D}}

\newcommand{\calX}{\mathcal{X}}
\newcommand{\calY}{\mathcal{Y}}

\newcommand{\calP}{\mathcal{P}}

\newcommand{\MMD}{\text{MMD}}
\newcommand{\eMMD}{\emph{MMD}}

\newcommand{\ID}{D_{I}}

\newtheorem{theorem}{Theorem}

\newtheorem{lemma}{Lemma}
\newtheorem{definition}{Definition}

\newtheorem{remark}{Remark}

\newtheorem{example}{Example}

\title{
Statistical Depth Meets Machine Learning: Kernel Mean Embeddings and Depth in Functional Data Analysis
}

\author[$\dagger$]{George Wynne}
\author[$\ddag$]{Stanislav Nagy}

\affil[$\dagger$]{
    Imperial College London, 
    Department of Mathematics,
    London, UK \newline
	{\small{\textit{E-mail:} \texttt{g.wynne18@imperial.ac.uk}}}}
         
\affil[$\ddag$]{
    Charles University,
	Faculty of Mathematics and Physics,
	Prague, Czech Rep. \newline
	{\small \textit{E-mail:} \texttt{nagy@karlin.mff.cuni.cz}}}

\date{\today}

\begin{document}

\makeatletter\let\Title\@title\makeatother
\maketitle

\begin{abstract}
Statistical depth is the act of gauging how representative a point is compared to a reference probability measure. The depth allows introducing rankings and orderings to data living in multivariate, or function spaces. Though widely applied and with much experimental success, little theoretical progress has been made in analysing functional depths. This article highlights how the common $h$-depth and related statistical depths for functional data can be viewed as a kernel mean embedding, a technique used widely in statistical machine learning. This connection facilitates answers to open questions regarding statistical properties of functional depths, as well as it provides a link between the depth and empirical characteristic function based procedures for functional data. 
\end{abstract}



\section{Introduction: Statistical Depth for Functional Data} 
Functional data analysis (FDA) concerns the study of observations that can be represented as functions, often residing in an infinite-dimensional space. In the recent decades, FDA saw remarkable progress, with many theoretical and practical problems successfully resolved. Often, however, statistical concepts used in finite-dimensional spaces do not readily generalise to random functions. As a result, alternative definitions and desired properties have to be used \citep{Ramsay1997, Ferraty_Vieu2006, horvath2012inference, Hsing2015}.  

An example of a prominent tool of multivariate analysis that is difficult to generalise to functional data is \emph{statistical depth}. Developed in the 1980s, statistical depth is an umbrella term for methods introducing orderings, ranks, and by extension, nonparametric statistical inference, to multivariate and more complex datasets. Let $\calX$ be a general topological sample space, and write $\calP(\calX)$ for the set of Borel probability measures on $\calX$; typical examples of $\calX$ are the Euclidean space ${\bbR}^d$ with $d \in \bbN$, or the space of square-integrable functions $L^{2}([0,1])$. Given $P\in\calP(\calX)$ a depth function $x \mapsto D(x;P)$, mapping $\calX$ to $[0,\infty)$, is meant to tell a user how representative of $P$ the point $x$ is. High values of the depth correspond to points $x$ located ``centrally" with respect to the given measure $P$; low depths flag atypical sample points, or outliers. An example of application of depth to two datasets is given in Figure~\ref{figure:depth}. Many depths have been proposed in $\calX = {\bbR}^d$; their comprehensive theory can be found in \citep{Zuo2000, Serfling2006} and an array of applications to multivariate analysis in \cite{Liu1999}. 

\begin{figure}[htpb]
    \centering
    \includegraphics[width=.45\linewidth]{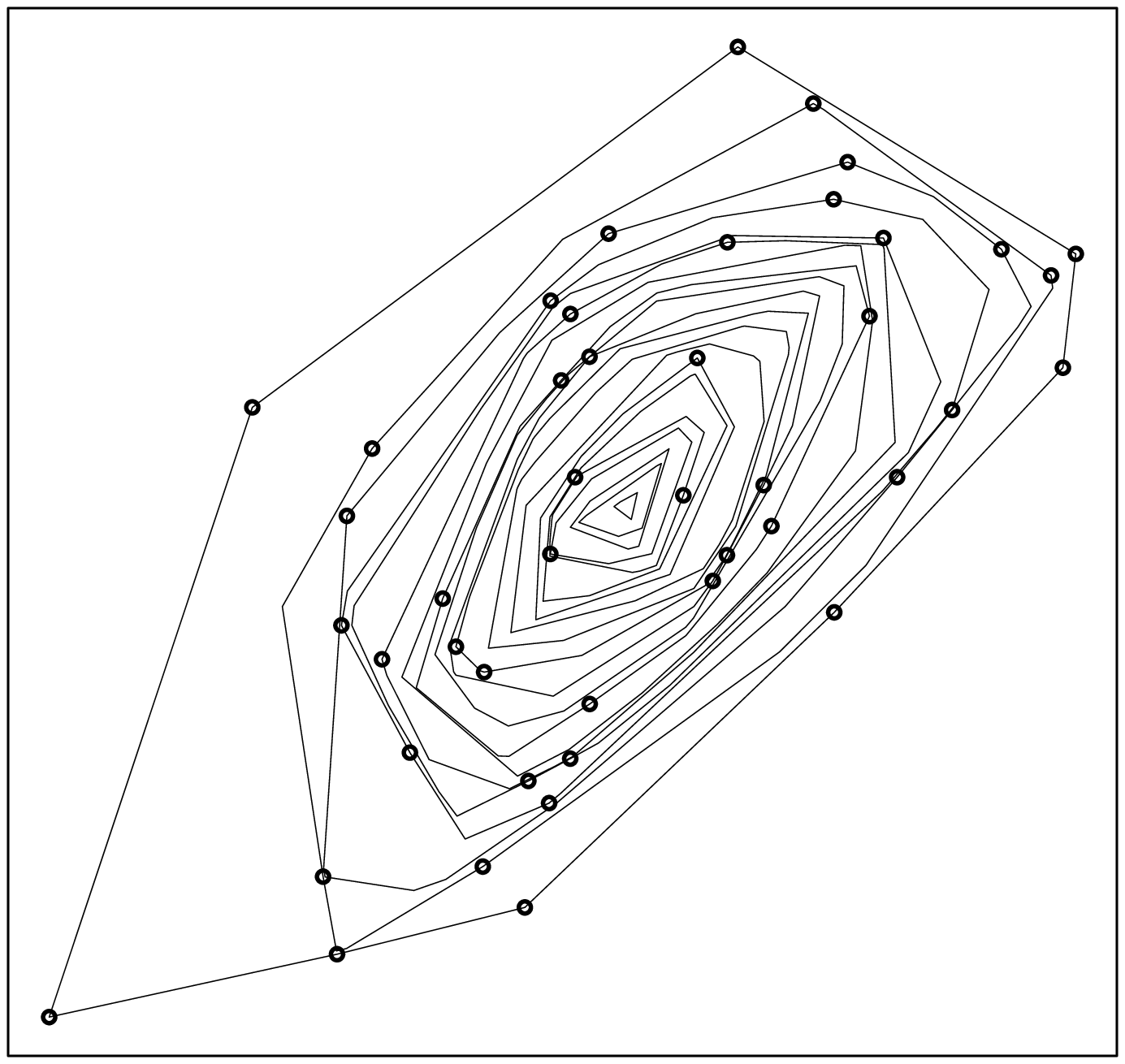} \quad \includegraphics[width=.45\linewidth]{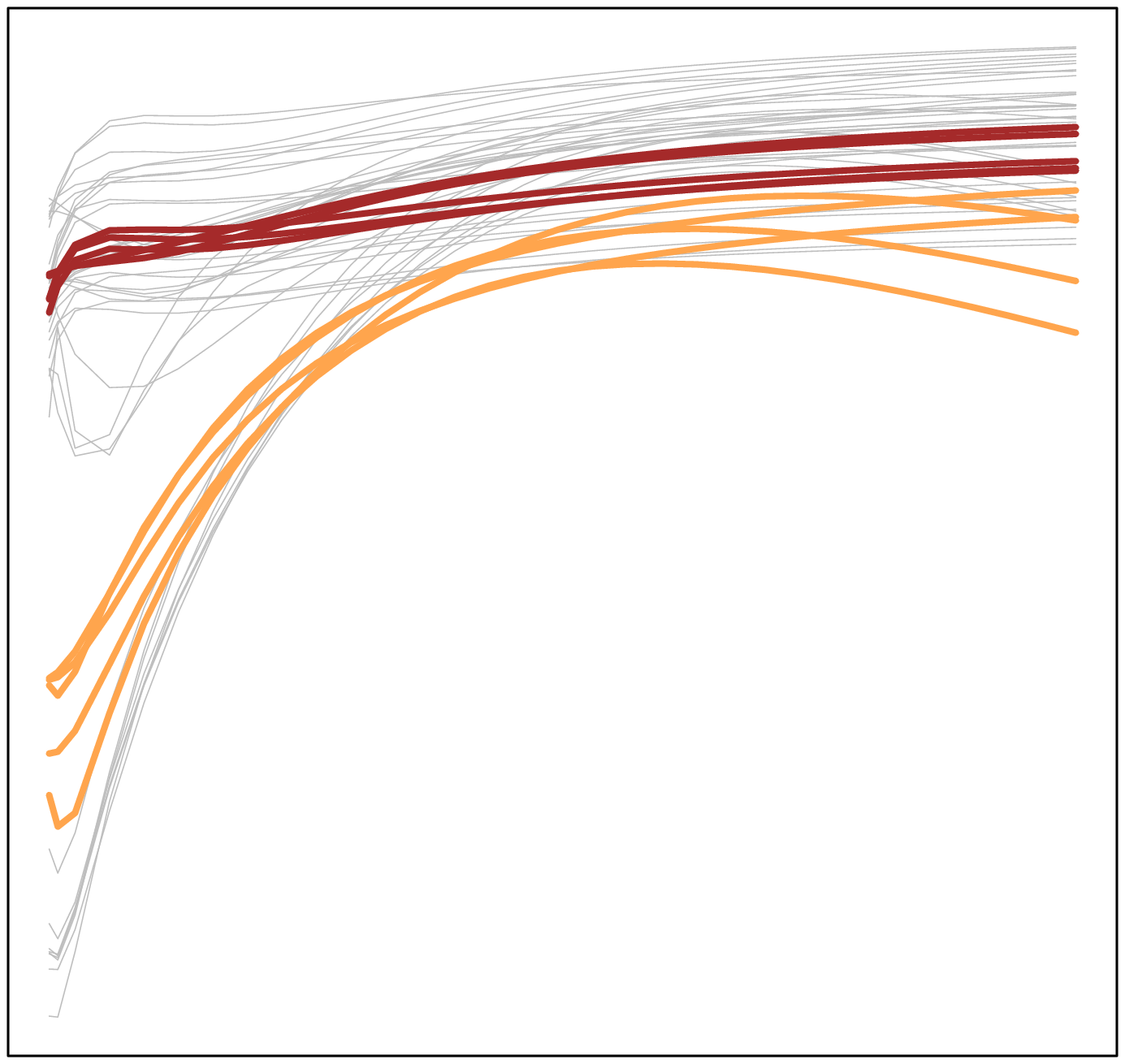}
    \caption{Left panel: Contours of halfspace depth \cite{Tukey1975} applied to a random sample of points in $\bbR^2$. The depth orders the sample points in terms of their centrality with respect to the geometry of the whole random sample. Right panel: A random sample of functional data $[0,1] \to \bbR$ with the observations attaining higher values of $h$-depth \cite{Cuevas2006} plotted in dark thick lines, and the data with low values of $h$-depth in light thick lines.}
    \label{figure:depth}
\end{figure}

Most of the standardly used finite-dimensional depths, such as the halfspace \cite{Tukey1975, Donoho_Gasko1992}, simplicial \cite{Liu1990}, or zonoid \cite{Koshevoy_Mosler1997, Mosler2002} depth are not possible to extend to function spaces directly. Instead, alternative approaches have been proposed in the literature \citep{Fraiman_Muniz2001, LopezPintado2009, Cuevas2009, NietoReyes2016, Gijbels2017}. These different notions of functional depths have found many applications in data visualisation \cite{Sun_Genton2011}, nonparametric estimation \cite{Fraiman_Muniz2001}, outlier detection \citep{Febrero2008}, classification \citep{LopazPintado2006}, or testing \citep{LopezPintado2009}, to name a few. 

Among functional depths, the \emph{$h$-depth} \citep{Cuevas2006} (also called the \emph{$h$-modal depth}, or the \emph{$h$-mode depth}) has emerged as a convenient choice which satisfies many of the desired criteria of a depth for functional inputs \citep{NietoReyes2016, Gijbels2017}. It is defined as follows.

\begin{definition}\label{def:h_mode}
Let $\calX$ be a vector space equipped with a norm $\norm{\cdot}_{\calX}$, $P\in\calP(\calX)$ and $\kappa\colon[0,\infty)\rightarrow[0,\infty)$ be a continuous, non-increasing function with $\kappa(0) > 0$ and $\lim_{t\rightarrow\infty}\kappa(t) = 0$. The \emph{$h$-depth} of $x\in\calX$ with respect to $P$ is defined as
\begin{align}
    D_{\kappa}(x;P) = \bbE_{X\sim P}\left[\kappa(\norm{x-X}_{\calX})\right].\label{eq:h_mode}
\end{align}
\end{definition}

A common choice of function $\kappa$ in \eqref{eq:h_mode} is $\kappa(t) = e^{-t^{2}/2}$, which gives rise to an $h$-depth that is sometimes referred to as the \emph{Gauss $h$-depth}.

\begin{remark}	\label{remark:bandwidth}
The $h$ in the name $h$-depth is from the rescaling $\kappa_{h}(\cdot) = h^{-1}\kappa(\cdot/h)$ used in practice in \eqref{eq:h_mode} instead of $\kappa$. The parameter $h>0$ plays the role of a bandwidth, but in this context is typically taken to be fixed to a constant rather than altering with new data \citep{Cuevas2007}. For ease of presentation we assume without loss of generality that $h=1$, for other values our exposition is identical with $\kappa_{h}$ replacing $\kappa$. 
\end{remark}

The sample $h$-depth takes a form similar to a kernel density estimator applied in a general normed vector space $\calX$. More precisely, let $X_1, \dots, X_n$ be a random sample of independent identically distributed (i.i.d.) variables generated from $P \in \calP(\calX)$, and let $P_{n} = n^{-1}\sum_{i=1}^{n}\delta_{X_{i}} \in \calP(\calX)$ be its random empirical measure. The sample $h$-depth \eqref{eq:h_mode} with respect to $P_n$ takes the form 
	\begin{equation}	\label{eq:sample h_depth}
	D_{\kappa}(x;P_{n}) = n^{-1}\sum_{i=1}^{n}\kappa(\norm{x-X_{i}}_{\calX}).	
	\end{equation}
From this expression we observe that the $h$-depth extends the notions of likelihood-based depths \cite{Fraiman_Meloche1999, Fraiman_etal1997} from $\bbR^d$ to functional data, and draws connections of the statistical depth with density-like concepts explored in function spaces \cite{Gasser_etal1998, Ferraty_Vieu2006}. The sample $h$-depth \eqref{eq:sample h_depth} is fast to compute and easily interpretable. Even though in $\calX = \bbR^d$ the concept does not fit directly into the framework of (global) statistical depths \cite{Zuo2000, Serfling2006} but rather to their localised counterparts \cite{Paindaveine_VanBever2013}, in function spaces it proved to be highly competitive \cite{Cuevas2006, Cuevas2007, Febrero_etal2007, Febrero2008, Chakraborty_Chaudhuri2014, NietoReyes2016, Pokotylo_Mosler2019}. The $h$-depth is frequently used as a well performing benchmark method in nonparametric FDA \cite{Gervini2012, Kuhnt_Rehage2016, Cuesta_etal2017}, and is available in standard FDA software packages \cite{R_fda.usc, ddalpha}. 

Despite its appeal in the practice of FDA, the theoretical properties of the $h$-depth in function spaces $\calX$ are largely unexplored \citep{Nagy2015, Nagy2016discrete, Gijbels2017} and many open questions remain \citep[Chapter 8]{Nagy2016}. These include consistency and asymptotics of $h$-depth estimators and the theory for testing based on this depth. A major open problem of functional depths is the characterisation of probability measures via the depth: Is the measure $P$ on $\calX$ uniquely characterised by its depth mapping $x \mapsto D(x;P)$? So far, no depth applicable to function spaces was proved to satisfy this highly desirable property \cite{Mosler_Mozharovskyi2021}.

This paper aims to summarise and leverage a surprising relationship between $h$-depth and \emph{kernel mean embeddings} (KME) to address these open questions, and shall use a mixture of existing and new results to do so. KMEs are well-studied in statistical machine learning, but they have not been utilised yet in the FDA literature. In the framework of KME, a kernel on $\calX$ is a symmetric positive definite function $k \colon \calX \times \calX \to \bbR$. Given a kernel $k$ a kernel mean embedding maps $P$ into the reproducing kernel Hilbert space (RKHS) \citep{Aronszajn1950} corresponding to $k$. That results in easy to manipulate expressions used to perform inference and tests for $\calX$-valued data \citep{Berlinet2004,Gretton2012,Sriperumbudur2010,Muandet2017}. Although not under the same name, KMEs have been studied for decades \citep{Guilbart1978} and in the last twenty years have become a successful statistical machine learning method \citep{Gretton05,Gretton2007} with applications to a wide range of data sources \citep{Gretton2012,Borgwardt2006}. Little attention has, however, been applied to functional data in the statistical machine learning literature. This paper aims to bridge the gap between FDA and statistical machine learning by applying KME, a statistical machine learning concept so far not used in FDA, to functional data.

The rest of this paper is structured as follows: In Section \ref{sec:KME} we introduce the kernel mean embeddings, RKHS and maximum mean discrepancy, and in Section~\ref{sec:Equivalance} we discuss the equivalence of those concepts from machine learning and the $h$-depth. A special class of integrated depths for functional data \cite{Cuevas2009, Nagy_etal2016}, including another popular functional depth called the \emph{(random) functional projection depth} \citep{Cuevas2007}, is also shown to be equivalent with KMEs in Section~\ref{subsec:integrated_depth}. Section \ref{sec:h_mode_KME} leverages the link between KME and $h$-depth to prove an array of new results regarding asymptotics and consistency of estimators of $h$-depth. Under minimal conditions we verify the uniform consistency and asymptotic normality of the sample version of $h$-depth in function spaces $\calX$, for both perfectly observed functional data, and discretised functional observations. Rates of convergence of the corresponding estimators are provided. Our results are remarkable because of their universality --- in contrast with both common functional depths \cite{Gijbels2017} and functional pseudo-density estimators \cite{Ferraty_Vieu2006}, the asymptotics of $h$-depth is shown to be valid without any restrictions on the probability measure $P$, and holds true on the whole infinite-dimensional sample space $\calX$. Results of this type are quite rare in nonparametric FDA, but as we show they follow directly from the related theory developed for KMEs. In Section \ref{sec:characterising} we demonstrate that under a mild condition, the $h$-depth characterises all probability measures in Hilbert spaces $\calX$. That result should be compared with the related advances in the theory of the depth in $\calX = \bbR^d$ \cite{Struyf_Rousseeuw1999, Mosler2002, Nagy2019c, Mosler_Mozharovskyi2021}. We show that the $h$-depth is the first statistical depth completely characterising not only all probability distributions in $\bbR^d$, but even in infinite-dimensional Hilbert spaces $\calX$. As such, the $h$-depth turns out to be of great interest in nonparametric FDA, where no natural and simply interpretable density-like concept characterising measures exists. In Section~\ref{sec:testing} we focus on statistical testing, and outline the equivalence between empirical characteristic function based tests, quite popular in FDA, and $h$-depth. Finally, Section~\ref{sec:Conclusion} provides concluding remarks. 


\section{Kernel Mean Embeddings}\label{sec:KME}
Kernel mean embeddings (KMEs) are an easily computable nonparametric method of reasoning with probability distributions.  Over the past 15 years, KMEs have become a commonplace tool in statistical machine learning. The earliest developments in the 1970s were done from a theoretical perspective, focusing on their topological properties \citep{Guilbart1978}; for a summary see \citep[Chapter 4]{Berlinet2004}. The first uses of KMEs within statistical machine learning were described in \citep{Gretton05,Gretton2007}, and they have found numerous applications since \citep{Sutherland2016,Briol2019,Muandet2017}. Kernel-based methods in general have been recently finding wider use also within FDA \citep{Kadri2016,Kupresanin2010,Berrendero2020}.

\subsection{Kernels and Reproducing Kernel Hilbert Spaces}
First of all, it is important to note that the definition of a kernel in machine learning is different from the definition of a kernel used in other areas of statistics, such as kernel smoothing or kernel density estimation. For example, in the FDA literature the function $\kappa$ in Definition~\ref{def:h_mode} is usually referred to as a kernel. Therefore, it is necessary to establish that the definition of a kernel we are using can be different from the definition that the reader may be familiar with before proceeding further.

\begin{definition}\label{def:kernel}
Let $\calX$ be a set. We call a function $k\colon \calX\times\calX\rightarrow \bbR$ a \emph{kernel} if it is \begin{enumerate*}[label=(\roman*)] \item symmetric, meaning $k(x,y) = k(y,x) \:\forall x,y\in\calX$, and \item positive definite, that is $\sum_{i,j=1}^{n}a_{i}a_{j}k(x_{i},x_{j})\geq 0$ for all $n\in\bbN$, $\{a_{i}\}_{i=1}^n\subset\bbR$ and $\{x_{i}\}_{i=1}^n \subset\calX$. \end{enumerate*}
\end{definition}

An example of a kernel over $\calX = \bbR^{d}$ is the squared exponential (also known as Gaussian) kernel $k(x,y) = e^{-\frac{1}{2}\norm{x-y}_{\bbR^{d}}^{2}}$. An analogous kernel can be defined on any normed space by replacing the Euclidean norm $\norm{\cdot}_{\bbR^{d}}$ with the norm on $\calX$. Before providing an array of other examples of kernels in Section \ref{subsec:examples}, we will now describe the motivation behind the definition of a kernel. The central idea is that a kernel is a measure of ``similarity" between its two inputs via an implicit inner product. We begin by describing how this occurs mathematically through the use of feature maps, and then comment on why kernels are useful for statistical machine learning tasks. 

\smallskip\noindent\textbf{Kernels as implicit inner products.} Inner products are a natural way to measure similarity between two inputs. For example, the standard inner product $\langle f,g\rangle_{L^{2}([0,1])}$ in $\calX = L^{2}([0,1])$ may be perceived as a measure of similarity --- for two functions $f,g \in L^2([0,1])$ of the same norm, the inner product $\langle f,g\rangle_{L^{2}([0,1])}$ is maximized if and only if $f = g$, thanks to the usual Cauchy-Schwarz inequality \cite[Corollary~5.1.4]{Dudley2002}. To obtain more complex measures of similarity one can transform an input before applying an inner product. Such a transform is known as using a feature map, or feature expansion, in machine learning \citep{Steinwart2008,ShaweTaylor2004}. Formally, if for a Hilbert space $\calH$ equipped with an inner product $\left\langle \cdot, \cdot \right\rangle_{\calH}$ and a function $\varphi\colon\calX\rightarrow\calH$ we have $k(x,y) = \langle \varphi(x),\varphi(y)\rangle_{\calH}$, then we say that $\varphi$ is a \emph{feature map} for $k$. 

\begin{example}
Let $\calX = \bbR$ and $k(x,y) = (xy+1)^{2}$. Setting $\calH = \bbR^{3}$ and $\varphi(x) = (x^{2},\sqrt{2}x,1)$ gives $k(x,y) = \langle \varphi(x),\varphi(y)\rangle_{\bbR^{3}}$. This example can straightforwardly be generalised to $\calX = \bbR^{d}$.
\end{example}

For any function $\varphi\colon\calX\rightarrow\calH$ with $\calH$ a Hilbert space, $k(x,y) = \langle \varphi(x),\varphi(y)\rangle_{\calH}$ is a kernel. The symmetry of $k$ follows from the symmetry of the inner product in $\calH$, and $\sum_{i,j=1}^{n}a_{i}a_{j}k(x_{i},x_{j}) = \langle\sum_{i=1}^{n}a_{i}\varphi(x_{i}),\sum_{j=1}^{n}a_{j}\varphi(x_{j})\rangle_{\calH}\geq 0$ for any $n \in \bbN$, $\{a_{i}\}_{i=1}^n \subset \bbR$ and $\{x_{i}\}_{i=1}^n \subset \calX$, so $k$ is positive definite. This begs the opposite question: Does every kernel have a feature map? The answer is affirmative \citep[Theorem 4.16]{Steinwart2008}, which justifies the statement that a kernel is a measure of similarity using an implicit inner product. 

\smallskip\noindent\textbf{Kernels in machine learning.} The reason this implicit representation is helpful is twofold. First, it is easier to check that a function is a kernel using Definition \ref{def:kernel} than to manually derive its feature map. Second, many numerical algorithms in machine learning, such as the least-squares estimator in linear regression, end up only depending on the inner products of the data. This provides the opportunity to ``kernelise" such algorithms by replacing the standard inner product with a kernel. Examples are the kernel ridge regression \citep{Steinwart2008}, support vector machines \citep{Taylor2005} and kernel ICA \citep{Bach2002}. The use of kernelised algorithms is equivalent to performing the algorithm on the $\calH$-valued data after the original $\calX$-valued data have been passed through the feature map $\varphi$. This explains the use of the terminology feature map, as it extracts features of the data which help to enhance the performance of the algorithm. Again, we do not need the feature map explicitly to do this, only the kernel.

\smallskip\noindent\textbf{Reproducing kernel Hilbert spaces.} For a given kernel $k$ the choice of $\calH$ and $\varphi$ is not unique. However, there is a canonical choice, called the reproducing kernel Hilbert space (RKHS) and a canonical feature map corresponding to the kernel.

\begin{definition}
A Hilbert space $\calH$ of functions from a set $\calX$ to $\bbR$ is called a \emph{reproducing kernel Hilbert space} if there exists a kernel $k$ on $\calX$ such that $k(\cdot,x)\in \calH$ for all $x\in\calX$ and 
	\begin{equation}	\label{eq:reproducing property}
	\langle f,k(\cdot,x)\rangle_{\calH} = f(x) \quad \mbox{for all $f\in \calH,x\in\calX$}.
	\end{equation}
The identity \eqref{eq:reproducing property} is called the \emph{reproducing property}.
\end{definition}

The Moore-Aronszajn theorem \cite{Aronszajn1950} guarantees there is a one-to-one relationship between kernels and RKHSs. This justifies why the RKHS of a kernel $k$ is considered a canonical choice of a feature space. Due to that relationship, it is common to use $\calH_{k}$ to denote the RKHS of $k$ and $\langle\cdot,\cdot\rangle_{k},\norm{\cdot}_{k}$ to denote the inner-product and the induced norm on $\calH_k$, respectively.

Setting $\varphi \colon \calX \to \calH_k, x \mapsto \varphi(x) = k(x,\cdot)$, meaning that $\varphi$ is a function-valued map, shows that $k(x,y) = \langle \varphi(x),\varphi(y)\rangle_{k}$ by the reproducing property \eqref{eq:reproducing property}. Thus, $\varphi$ is the canonical feature map. Using the Cauchy-Schwarz inequality \cite[Corollary 5.1.4]{Dudley2002} along with the reproducing property \eqref{eq:reproducing property} is a common trick for obtaining bounds for functions in an RKHS. Namely, for any $x \in \calX$ and $f \in \calH_k$ we have 
\begin{align}
    \lvert f(x)\rvert = \lvert \langle f,k(x,\cdot)\rangle_{k}\rangle \leq \norm{f}_{k}\norm{k(x,\cdot)}_{k} = \norm{f}_{k}\sqrt{k(x,x)}.\label{eq:cauchy_scwhwarz}
\end{align} This formula shows that the RKHS norm controls the supremum norm of $f$ if $k(x,x)$ is bounded for $x \in \calX$, which is true for nearly all kernels commonly used in practice. 

Explicit forms for RKHS when $\calX = \bbR^{d}$ are known for a wide range of kernels. Often, spectral properties of the kernels are used to obtain the characterisations \citep{Berlinet2004,Kanagawa2018Review}. When $\calX$ is non-Euclidean, it is harder to derive interpretable representations of RKHSs, but results in this direction for $\calX$ a separable Hilbert space and a particular family of kernels are provided in \citep[Theorem 2, Proposition 2]{Wynne2020a} and \citep{nelsen2020random}.

The reproducing property \eqref{eq:reproducing property} allows one to easily manipulate quantities such as integrals and expectations of functions from $\calH_k$ using the kernel, which would be difficult in other function spaces. As a result, RKHSs are widely studied in statistical machine learning since they facilitate the analysis of many machine learning algorithms involving kernels \citep{Hofmann2008,ShaweTaylor2004,Steinwart2008}. For a summary of the basic properties of RKHSs with some applications to statistics see \citep{Steinwart2008,Berlinet2004,Paulsen2016}. 

\subsection{Examples of Kernels}\label{subsec:examples}

For $\calX = \bbR^{d}$ the most basic kernel is the standard Euclidean inner product kernel $k(x,y) = \langle x,y\rangle_{\bbR^{d}}$. The corresponding RKHS is the set of functions of the form $f_{y}(x) = \langle x,y\rangle_{\bbR^{d}}$ for some $y\in\bbR^{d}$ and the inner product of the RKHS is $\langle f_{x},f_{y}\rangle_{k} = \langle x,y\rangle_{\bbR^{d}}$. 

A kernel on $\bbR^{d}$ commonly used for spatial modelling \citep{Stein1999} is the Mat\'ern-$3/2$ kernel $k(x,y) = (1+\sqrt{3}\norm{x-y}_{\bbR^{d}})e^{-\sqrt{3}\norm{x-y}_{\bbR^{d}}}$ which is a particular instance of the wider Mat\'ern class \citep{Berlinet2004}. The RKHS of such kernels are Sobolev spaces \citep{Kanagawa2018Review} of functions that map from subsets of $\bbR^{d}$ to $\bbR$. 

\begin{remark}	\label{remark:bandwidth for kernel}
Note that given any kernel $k$ one can generate classes of kernels by introducing hyper-parameters. The most typical is a bandwidth obtained by replacing $k(x,y)$ with $k(x/\gamma,y/\gamma)$ for some $\gamma > 0$. As we will see in Section~\ref{sec:Equivalance}, this corresponds to taking different values of the tuning parameter $h > 0$ in $h$-depth ass discussed in Remark~\ref{remark:bandwidth}.
\end{remark}

Two kernels which are popular for use in kernel mean embeddings, the main topic of this paper to be introduced in Section \ref{subsec:KME}, are the \emph{squared exponential} (SE) and \emph{inverse multi-quadric} (IMQ) kernels. We will present them now in the generality required to deal with inputs from function spaces \citep{Wynne2020a}. Let $\calX,\calY$ be Hilbert spaces and $T\colon\calX\rightarrow\calY$. The SE-$T$ kernel is
	\begin{equation}	\label{eq:SE kernel}
	k_{\text{SE-$T$}}(x,y) = e^{-\frac{1}{2}\norm{T(x)-T(y)}_{\calY}^{2}},	
	\end{equation}
and the IMQ-$T$ kernel is 
	\begin{equation}	\label{eq:IMQ kernel}
	k_{\text{IMQ-$T$}}(x,y) = (\norm{T(x)-T(y)}_{\calY}^{2}+1)^{-1/2}.	
	\end{equation}
In the special case $\calX = \calY$ and $T = \gamma^{-1}I$ with $I$ the identity map and $\gamma > 0$ we recover the bandwidth hyper-parameter from Remark~\ref{remark:bandwidth for kernel}. For $\calX = \bbR^{d}$ the SE-$I$ kernel is also referred to as the Gauss kernel. Its RKHS has been investigated in \citep{Minh2009,Steinwart2006}. For different choices of $T$ the RKHS of the SE-$T$ kernel in more general spaces $\calX$ was studied in \citep[Theorem 2]{Wynne2020a}.

\smallskip\noindent\textbf{Translation invariant kernels.} The important collection of \emph{translation invariant kernels} in a Hilbert space $\calX$, meaning $k(x,y) = \phi(x-y)$ for some $\phi \colon \calX \to \bbR$, has an intimate connection with the Fourier transform of probability measures. Recall that for $\mu\in\calP(\calX)$, the Fourier transform of $\mu$ is defined as $\widehat{\mu}(x) = \int_{\calX}e^{i\langle x,v\rangle_{\calX}}d\mu(v)$. In probability theory, the Fourier transform is also known as the characteristic function of measure $\mu$. When $\calX$ is finite-dimensional, Bochner's theorem \citep{Bochner1959} states that a kernel is translation invariant with $\phi(0) = 1$ and continuous if and only if $k(x,y) = \widehat{\mu}(x-y)$ for some $\mu\in\calP(\calX)$. For example, in $\calX = \bbR^{d}$ we have $k_{\text{SE-$I$}}(x,y) = \widehat{N}_{I_{d}}(x-y)$ for $N_{I_{d}}$ the standard $d$-variate normal distribution.

When $\calX$ is an infinite-dimensional Hilbert space, Bochner's theorem cannot be applied in the description of translation invariant kernels. Instead, the more involved Minlos-Sazonov theorem must be used, see e.g.\ \citep[Theorem VI.1.1]{Vakhania1987} or \citep[Theorem 1.1.5]{Maniglia2004}. To give a simple infinite-dimensional example define $L^{+}_{1}(\calX)$ to be the space of symmetric, positive, trace class operators from $\calX$ to $\calX$ \citep{DaPrato2006}. Then for $C\in L^{+}_{1}(\calX)$ we have 
	\begin{equation}	\label{eq:SE relation}
	k_{\text{SE-$C^{1/2}$}}(x-y) = \widehat{N}_{C}(x-y),
	\end{equation}
where $N_{C}$ is the Gaussian measure on $\calX$ with mean zero and covariance operator $C$, and $C^{1/2}$ is square root operator of $C$. 

\subsection{Definition of Kernel Mean Embedding}\label{subsec:KME}
We are now ready to define the kernel mean embeddings, the main object of study of this paper. It will allow us to draw connections to $h$-depth and empirical characteristic function based testing procedures. 

\begin{definition}\label{def:KME}
Suppose $\calX$ is a measurable space, $k$ a measurable kernel, and define $\calP_{k}\subseteq\calP(\calX)$ as $\{P\in\calP(\calX)\colon \int_{\calX}\sqrt{k(x,x)}dP(x)<\infty\}$. For $P\in\calP_{k}$ define the \emph{kernel mean embedding} $\Phi_{k}P$ of $P$ as
\begin{align}
    \Phi_{k}P\coloneqq \int_{\calX} k(\cdot,x)dP(x).\label{eq:KME}
\end{align}
\end{definition}

A few comments are in order. The integral in \eqref{eq:KME} is to be understood as a Bochner integral \citep[Chapter 2]{Hsing2015}. Therefore, $\Phi_{k}P$ is an element of $\calH_{k}$, meaning it is a function from $\calX$ to $\bbR$. It represents the linear operator from $\calH_{k}$ to $\bbR$ given by $f\mapsto \int_{\calX}f(x)dP(x)$. This linear operator is bounded since
\begin{align}
    \left\lvert\int_{\calX}f(x)dP(x)\right\rvert \leq \int_{\calX}\lvert f(x)\rvert dP(x)  & = \int_{\calX}\lvert\langle f,k(\cdot,x)\rangle_{k}\rvert dP(x)\nonumber\\
    & \leq \norm{f}_{k}\int_{\calX}\sqrt{k(x,x)}dP(x)\leq c_{k}\norm{f}_{k},\label{eq:CS_RKHS}
\end{align}for some $c_{k}<\infty$ by the assumption on $P$, where \eqref{eq:CS_RKHS} is by \eqref{eq:cauchy_scwhwarz}. The interpretation of $\Phi_{k}P$ as a Bochner integral allows us to write for $f\in\calH_{k}$
\begin{align}
    \bbE_{X\sim P}[f(X)] = \bbE_{X\sim P}[\langle f,k(\cdot,X)\rangle_{k}] = \langle f,\bbE_{X\sim P}[k(\cdot,X)]\rangle_{k} = \langle f,\Phi_{k}P\rangle_{k},\label{eq:bochner_exp}
\end{align}where the first equality is by the reproducing property \eqref{eq:reproducing property} and the second by the definition of the Bochner integral \cite{Gine2015}.

How much smaller $\calP_{k}$ is compared to $\calP(\calX)$? It turns out that $\calP_{k} = \calP(\calX)$ if and only if $k$ is bounded \citep[Proposition 2]{Sriperumbudur2010}, which is satisfied by nearly all kernels used in practice. The integral $\Phi_{k}P$ is empirically estimated given i.i.d.\ samples $\{X_{i}\}_{i=1}^{n}$ from $P$ by replacing $P$ with the empirical measure $P_{n} = n^{-1}\sum_{i=1}^{n}\delta_{X_{i}}$. We obtain
	\begin{equation}	\label{eq:sample KME}
	\Phi_{k}P_{n} = n^{-1}\sum_{i=1}^{n}k(\cdot,X_{i}). 
	\end{equation}
The discussion so far has assumed little structure on $\calX$; indeed we could take $\calX$ to be infinite-dimensional if one so desires.

Closed form expressions of KME for certain pairs of kernels and probability distributions exist. The most relevant to FDA are closed forms when $P$ is a Gaussian measure on $\calX$, i.e.\ a Gaussian process whose paths lie in $\calX$. Such formulas can be found in \citep{Kellner2019,Wynne2020a} for the SE-$T$ kernel from \eqref{eq:SE kernel}, see Example \ref{exp:mode_rates}. Closed form expressions in the case $\calX\subseteq\bbR^{d}$ can be found in \citep{Briol2019_integration}.

\subsection{Maximum Mean Discrepancy}
Armed with the definition of a KME, one may ask for $P,Q\in\calP_{k}$, how different $\Phi_{k}P$ and $\Phi_{k}Q$ are as functions in $\calH_k$? This is measured by the maximum mean discrepancy (MMD) defined as the norm in the RKHS of the difference of the two embeddings.
\begin{definition}
Given a measurable space $\calX$, a kernel $k$ on $\calX$ and $P,Q\in\calP_{k}(\calX)$ the \emph{maximum mean discrepancy} between $P$ and $Q$ is defined as $\eMMD_{k}(P,Q) = \norm{\Phi_{k}P-\Phi_{k}Q}_{k}$.
\end{definition}

Due to the reproducing property \eqref{eq:reproducing property} there are other interpretations of MMD in terms of the kernel. They will help us better understand how MMD behaves and how it depends on the properties of the kernel $k$. The next result is well known in the KME literature \citep{Sriperumbudur2010,Gretton2012} and we provide the proof in full for clarity.

\begin{theorem}\label{thm:MMD_reps}
Let $\calX$ be a topological space, $k$ a kernel on $\calX$ and $P,Q\in\calP_{k}(\calX)$. Then 
\begin{align}
    \eMMD_{k}(P,Q)^{2} & = \left(\sup_{\norm{f}_{k}\leq 1}\left\lvert\int_{\calX}f(x)dP(x) - \int_{\calX}f(x)dQ(x)\right\rvert\right)^{2}\label{eq:IPM}\\
    & = \int_{\calX}\int_{\calX}k(x,y)d(P-Q)(x)d(P-Q)(y). \label{eq:double_integral}
\end{align}
If $\calX$ is a Hilbert space and $k(x,y) = \widehat{\mu}(x-y)$ for some $\mu\in\calP(\calX)$ then
\begin{align}
    \eMMD_{k}(P,Q)^{2} = \int_{\calX}\lvert\widehat{P}(v)-\widehat{Q}(v)\rvert^{2}d\mu(v).\label{eq:L2}
\end{align}
\end{theorem}
\begin{proof}
First, by \eqref{eq:bochner_exp}
\begin{align*}
    \sup_{\norm{f}_{k}\leq 1}\left\lvert\int_{\calX}f(x)dP(x) - \int_{\calX}f(x)dQ(x)\right\rvert & = \sup_{\norm{f}_{k}\leq 1}\left\lvert \langle f,\Phi_{k}P-\Phi_{k}Q\rangle_{k}\right\rvert\\
    & = \norm{\Phi_{k}P-\Phi_{k}Q}_{k},
\end{align*} where the second equality is by Cauchy-Schwarz \cite[Corollary 5.1.4]{Dudley2002}. For the second representation of $\text{MMD}_{k}$,
\begin{align*}
    \norm{\Phi_{k}P-\Phi_{k}Q}_{k}^{2} = \langle \Phi_{k}P,\Phi_{k}P\rangle_{k} + \langle \Phi_{k}Q,\Phi_{k}Q\rangle_{k} - 2\langle \Phi_{k}P,\Phi_{k}Q\rangle_{k},
\end{align*}
so it suffices to compute $\langle\Phi_{k}P,\Phi_{k}Q\rangle_{k}$, and then set $P=Q$ for the other two terms. To this end, using \eqref{eq:bochner_exp} on $\Phi_{k}P$ gives 
\begin{align*}
    \langle\Phi_{k}P,\Phi_{k}Q\rangle_{k} = \bbE_{Y\sim Q}[\Phi_{k}P(Y)] = \int_{\calX}\int_{\calX}k(x,y)dP(x)dQ(y),
\end{align*}
which completes the proof of the second representation. For the last representation we write
\begin{align}
    \MMD_{k}(P,Q)^{2} & = \int_{\calX}\int_{\calX}k(x,y)d(P-Q)(x)d(P-Q)(y)\label{eq:quad}\\
    & = \int_{\calX}\int_{\calX}\int_{\calX}e^{i\langle x-y,v\rangle_{\calX}}d\mu(v)d(P-Q)(x)d(P-Q)(y)\label{eq:FT_formula}\\
    & = \int_{\calX}\left(\int_{\calX}e^{i\langle x,v\rangle_{\calX}}d(P-Q)(x)\right)\left(\int_{\calX}e^{-i\langle y,v\rangle_{\calX}}d(P-Q)(y)\right)d\mu(v)\label{eq:fubini}\\
    & = \int_{\calX}\lvert\widehat{P}(v)-\widehat{Q}(v)\rvert^{2}d\mu(v),\label{eq:conj}
\end{align} where \eqref{eq:quad} is using \eqref{eq:double_integral}, \eqref{eq:FT_formula} is by the definition of the Fourier transform for measures, \eqref{eq:fubini} is Fubini's theorem \cite[Theorem~4.4.5]{Dudley2002} and \eqref{eq:conj} comes from the product of a complex number and its conjugate being its absolute value squared.
\end{proof}

In Section \ref{subsec:examples} we provided examples of kernels that satisfy $k(x,y) = \widehat{\mu}(x-y)$ for some $\mu\in\calP(\calX)$. Quantities of the form \eqref{eq:IPM} are known as integral probability metrics \citep{muller1997integral}. Many common distances of probability measures are of this form --- prominent examples are the Wasserstein and the total variation distance, where the only difference is that the supremum in \eqref{eq:IPM} is taken over a different set. MMD can be easily estimated given samples from $P,Q$ by means of a standard empirical estimate of \eqref{eq:double_integral}. This approximation is very straightforward and has pleasant statistical properties. The representation \eqref{eq:L2} shows how MMD can be viewed as a weighted $L^{2}$-type distance between the Fourier transforms of probability measures. This will allow us to draw connections between MMD and testing procedures based on the empirical characteristic function in Section~\ref{sec:testing}.

\section{Equivalence of Kernel Mean Embeddings and \texorpdfstring{$h$}{h}-Depth}\label{sec:Equivalance}
An immediate connection between KME and $h$-depth can be made. It will be used in Section \ref{sec:h_mode_KME} to establish consistency and asymptotic normality of $h$-depth estimators in a streamlined way using the structure of the RKHS. In Section \ref{sec:characterising} we use it to show that $h$-depth characterises probability distributions completely. Although some instances of depth functions have already been used in machine learning research \citep{Gilad_Burges2013, Dutta_etal2016, Kleindessner_Luxburg2017}, as far as we are aware, this is the first solid connection of a depth notion with concepts from theoretical machine learning literature. The next theorem connects KME and $h$-depth. Its statement follows directly from Definitions~\ref{def:h_mode} and~\ref{def:KME}. 

\begin{theorem}\label{thm:h_mode_equivalence}
If $\calX$ is a normed vector space and 
	\begin{equation}	\label{eq:kernel by kappa}
	k(x,y) = \kappa(\norm{x-y}_{\calX})
	\end{equation}
is a kernel on $\calX$ with $\kappa$ satisfying the conditions in Definition \ref{def:h_mode}, then $D_{\kappa}(x;P) = \Phi_{k}P(x)$.
\end{theorem}

It is natural to ask what conditions are needed on $\kappa$ aside from those in Definition \ref{def:h_mode} to ensure that the corresponding $k$ defined in \eqref{eq:kernel by kappa} is a kernel. This is in fact a classical problem studied since the 1930s --- kernels of the form \eqref{eq:kernel by kappa} are known as \emph{radial kernels} \citep{Schoenberg1938} since they depend only on the normed difference of the two inputs. The next result was proven in \citep[Theorem 1.1]{Scovel2010} in the case of $\calX = \bbR^{d}$. But, the same proof works for $\calX$ a separable Hilbert space since the supporting results \citep{Schoenberg1938} also hold in that situation. Following \citep{Schoenberg1938}, call a function $f\colon[0,\infty)\rightarrow[0,\infty)$ \emph{completely monotone} if $(-1)^{n}f^{(n)}(t) \geq 0$ for all $t>0$ and all $n=0,1,2,\dots$.

\begin{theorem}\label{thm:radial}
Let $\calX$ be a separable Hilbert space, $\kappa\colon[0,\infty)\rightarrow[0,\infty)$ and $k(x,y) = \kappa(\norm{x-y}_{\calX})$. Then the following are equivalent:
\begin{enumerate}[label=(\roman*)]
    \item $k$ is a kernel.
    \item There exists a finite Borel measure $\mu$ on $[0,\infty)$ such that
    \begin{align}
        k(x,y) = \int_{0}^{\infty}e^{-t^{2}\norm{x-y}_{\calX}^{2}}d\mu(t).\label{eq:laplace}
    \end{align}
    \item $\kappa(\sqrt{\cdot})$ is completely monotone.
\end{enumerate}
\end{theorem}

If $\kappa(\sqrt{\cdot})$ is completely monotone then $\kappa$ is continuous and non-increasing. By \eqref{eq:laplace} one can see that if $\kappa$, or equivalently $k$, is zero anywhere, then $\mu$ must be the zero measure hence also $\kappa$ would be identically zero. Therefore if $\kappa$ is not identically zero then it is zero nowhere and $\kappa(0) > 0$. This leaves only the decay condition in Definition \ref{def:h_mode} to be checked on a case by case basis to conclude that a radial kernel corresponds to a function $\kappa$ that satisfies Definition \ref{def:h_mode}. Not all kernels decay to zero, as can be seen in the example of a kernel $k(x,y) = e^{-\norm{x-y}_{\bbR^{d}}^{2}}+1$ on $\bbR^d$. Nevertheless, if the condition of decay to zero in Definition \ref{def:h_mode} is relaxed to 
	\begin{equation}	\label{eq:modified decay}
	\lim_{t\rightarrow \infty}\kappa(t) = \inf_{t\geq 0}\kappa(t)
	\end{equation}
then $\kappa(\sqrt{\cdot})$ being completely monotone ensures this modified decay condition holds. 

In conclusion, if condition \eqref{eq:modified decay} is used in Definition~\ref{def:h_mode} then for any radial kernel $k$ and $\kappa$ defined by \eqref{eq:kernel by kappa} the KME corresponds to an $h$-depth. Note that interestingly \eqref{eq:laplace} shows that any radial kernel is an average of squared exponential kernels with the average being taken over the bandwidth hyper-parameter considered in Remark~\ref{remark:bandwidth for kernel}. 

\smallskip\noindent\textbf{Examples of $h$-depths corresponding to KMEs.}
Examples of functions $\kappa$ that satisfy the conditions of Theorem~\ref{thm:radial} include $\kappa(t) = e^{-t^{2}/2}$ corresponding to the SE-$I$ kernel introduced in \eqref{eq:SE kernel} in Section \ref{subsec:examples}. This is a standard choice in the practice of depth-based methods resulting in the widely used Gauss $h$-depth function. Another example not explored in the literature on $h$-depth is $\kappa(t) = (t^{2}+1)^{-1/2}$ corresponding to the IMQ-$I$ kernel \eqref{eq:IMQ kernel}. In fact, the IMQ-$I$ kernel also appears naturally in the context of statistical analysis as it can be obtained from \eqref{eq:laplace} by setting $\mu$ to be a Gaussian measure restricted to $[0,\infty)$ with appropriate scaling. In more generality, if $T \colon \calX \to \calY$ is a linear map then by Theorem \ref{thm:radial} both the SE-$T$ and IMQ-$T$ kernels are $h$-depths. Indeed, in that situation $k(x,y) = \kappa(\norm{x-y}_{\mathcal{Z}})$ where $\mathcal{Z}$ is the Hilbert space with inner product $\langle x,y\rangle_{\mathcal{Z}} = \langle Tx,Ty\rangle_{\calY}$. 

\subsection{Integrated \texorpdfstring{$h$}{h}-Depth and (Random) Functional Projection \texorpdfstring{$h$}{h}-Depth}\label{subsec:integrated_depth}

Besides the $h$-depths considered throughout this article, another major class of functional depths widely used in practice are the \emph{integrated depths} \citep{Fraiman_Muniz2001, Cuevas2009, Nagy_etal2016, Ramsay_etal2019}. Special cases of integrated depths are the popular modified band depth \cite{LopezPintado2009} and the multivariate functional halfspace depth \cite{Claeskens_etal2014}. Integrated depths are in general obtained by averaging univariate depths of one-dimensional projections of functional data, with respect to a reference measure defined in the dual of the function space $\calX$. For simplicity we will focus on the situation when $\calX$ is a Hilbert space and consider integrated depths of the form 
	\[	\ID(x;P) = \int_{\calX}D_{\kappa_{1}}(\langle x,v\rangle_{\calX};P_{v})d\nu(v)	\] 
where $D_{\kappa_{1}} \colon \bbR \times \calP(\bbR) \to [0,\infty)$ is a one-dimensional $h$-depth \eqref{eq:h_mode} using $\kappa_{1}\colon [0,\infty) \to [0,\infty)$, $P_{v}\in\calP(\bbR)$ is the distribution of $\langle X,v\rangle_{\calX}$ where $X\sim P$ and $v \in \calX$, and $\nu$ is some measure on $\calX$ (identified with its dual using the Riesz representation theorem \citep[Theorem~3.2.1]{Hsing2015}). In this context, we may write
\begin{align}
    \ID(x;P) = \int_{\calX}D_{\kappa_{1}}(\langle x,v\rangle_{\calX};P_{v})d\nu(v) = \int_{\calX}\int_{\calX}\kappa_{1}\left(\lvert\langle x,v\rangle_{\calX}-\langle y,v\rangle_{\calX}\rvert\right)dP(y)d\nu(v).\label{eq:integrated_depth}
\end{align}
A depth quite similar to \eqref{eq:integrated_depth} was used also in \citep{Cuevas2007} under the name \emph{random projection depth}. The term \emph{random} comes from the way this depth is used in practice, as depths of this type typically have to be numerically approximated. The outer integral in \eqref{eq:integrated_depth} is commonly replaced by an average taken over i.i.d.\ realisations $v_1, \dots, v_m \in \calX$ sampled from $\nu$, for $m \in \bbN$ large enough. This \begin{enumerate*}[label=(\roman*)] \item makes the practically used integrated depths inherently random, \item can be computationally expensive if $\nu$ is hard to sample from, and \item necessarily decreases the accuracy of procedures which rely on the depth, and the computation of the depth itself. \end{enumerate*} The next theorem gives conditions under which an integrated $h$-depth \eqref{eq:integrated_depth} can be represented as a KME. In particular, we demonstrate that under mild conditions the depth in \eqref{eq:integrated_depth} can be expressed in a closed form, without the need for a random numerical approximation of the outer integral. Hence, the term \emph{random} can be dropped.

\begin{theorem}\label{thm:integrated_depth}
Let $\calX$ be a Hilbert space and $\ID$ be an integrated $h$-depth of the form \eqref{eq:integrated_depth} for $\kappa_1$ satisfying Theorem~\ref{thm:radial}. Then there exists $\mu_1 \in \calP(\bbR)$ whose Fourier transform $\widehat{\mu}_{1}$ satisfies $\kappa_{1}(\lvert s-t\rvert) = \kappa_1(0) \widehat{\mu}_{1}(s-t)$ for all $s, t \in \bbR$. In addition, $\ID(x;P) = \kappa_1(0)\,\Phi_{k}P(x)$ for $k(x,y) = \widehat{\mu}(x-y)$. Here, $x,y \in \calX$, and $\mu\in\calP(\calX)$ is the probability measure associated with the random element $UV$ in $\calX$ where $U\sim \mu_{1} \in \calP(\bbR)$ and $V\sim\nu \in \calP(\calX)$ are independent.
\end{theorem}
\begin{proof}
The existence of $\mu_1$ follows from Theorem~\ref{thm:radial} and our remarks on translation invariant kernels from Section~\ref{subsec:examples} applied to function $\widetilde{\kappa}(\cdot) = \kappa_1(\cdot)/\kappa_1(0)$. Denote by $k_1$ the kernel on $\bbR$ corresponding to $\widetilde{\kappa}$. Starting at \eqref{eq:integrated_depth} we can write
\begin{align}
    \ID(x;P) & = \kappa_1(0) \, \int_{\calX}\int_{\calX}\widetilde{\kappa}\left(\lvert\langle x,v\rangle_{\calX}-\langle y,v\rangle_{\calX}\rvert\right)dP(y)d\nu(v)\nonumber\\
    & = \kappa_1(0) \, \int_{\calX}\int_{\calX}k_{1}(\langle x,v\rangle_{\calX},\langle y,v\rangle_{\calX})dP(y)d\nu(v)\label{eq:k_1}\\
    & = \kappa_1(0) \, \int_{\calX}\int_{\calX}\int_{\bbR} e^{i\langle x-y,uv\rangle_{\calX}}d\mu_{1}(u)dP(y)d\nu(v)\label{eq:expand_k_1}\\
    & = \kappa_1(0) \, \int_{\calX}\int_{\calX}\int_{\bbR}e^{i\langle x-y,u v\rangle_{\calX}}d\mu_{1}(u)d\nu(v)dP(y)\label{eq:fubini_swap}\\
    & = \kappa_1(0) \, \int_{\calX}k(x,y)dP(y) = \kappa_1(0) \, \Phi_{k}P(x),\label{eq:Phi_k}
\end{align}
where \eqref{eq:k_1} is by the assumption on $\kappa_{1}$, \eqref{eq:expand_k_1} is by the assumption $k_{1}(s,t) = \widehat{\mu}_{1}(s-t)$, \eqref{eq:fubini_swap} is by Fubini's theorem \cite[Theorem~4.4.5]{Dudley2002} and \eqref{eq:Phi_k} is by noting that $\int_{\calX}\int_{\bbR}e^{i\langle x-y,uv\rangle_{\calX}}d\mu_{1}(u)d\nu(v)$ is the Fourier transform of the random variable $UV$ with $U\sim\mu_1$ and $V\sim\nu$ are independent. 
\end{proof}

A common choice of $\nu$ in practice is a Gaussian measure on $\calX$. The next example shows that the framework of Theorem \ref{thm:integrated_depth} provides a way to obtain the integrated depth \eqref{eq:integrated_depth} in a closed form, circumventing the need for an empirical approximation as previously done in the literature \cite{Cuevas2007, Febrero2008, Claeskens_etal2014}. This important observation brings an original viewpoint on the standardly used (random) functional projection depths, and explains their good performance consistently observed in simulations and real data applications. It also allows direct computation of the sample version of the integrated depth \eqref{eq:integrated_depth} by means of the empirical KME as considered in \eqref{eq:sample KME}. 

\begin{example}\label{exp:integrated_depth}
Consider $\ID(x;P)$ from \eqref{eq:integrated_depth} with $\kappa_{1}(\lvert s-t\rvert) = k_{1}(s,t) = e^{-\frac{1}{2}\left\vert s-t\right\vert^{2}}$, so $\mu_{1}$ is the standard normal on $\bbR$ and $\nu = N_{C}$ for some Gaussian measure with mean zero and covariance operator $C\in L^{+}_{1}(\calX)$. Then the corresponding kernel $k$ from Theorem \ref{thm:integrated_depth} is equal to the IMQ-$C^{1/2}$ kernel from \eqref{eq:IMQ kernel} from Section \ref{subsec:examples}. To see this first note that the IMQ-$C^{1/2}$ kernel may be written as
\begin{align}
    k_{\text{IMQ-$C^{1/2}$}}(x,y) = \left(\norm{C^{1/2}x-C^{1/2}y}_{\calX}^{2}+1\right)^{-1/2}= \int_{\calX}e^{-\frac{1}{2}\langle x-y,v\rangle_{\calX}^{2}}dN_{C}(v),\label{eq:IMQ}
\end{align}
due to the distribution of $\langle x-y,v\rangle_{\calX}$ with $v\sim N_{C}$ being the univariate normal distribution $N(0,\langle C(x-y),x-y\rangle_{\calX})$ and standard Gaussian integral formulas, for a proof see \citep[Section 5]{Wynne2020a}. Now using \eqref{eq:IMQ} along with \eqref{eq:k_1} and Fubini's theorem \cite[Theorem~4.4.5]{Dudley2002} we obtain
\begin{align*}
    \ID(x;P) & = \int_{\calX}\int_{\calX}k_{1}(\langle x,v\rangle_{\calX},\langle y,v\rangle_{\calX})dN_{C}(v)dP(y) \\
    & = \int_{\calX}k_{\text{IMQ-$C^{1/2}$}}(x,y)dP(y) = \Phi_{k_{\text{IMQ-$C^{1/2}$}}}P(x).
\end{align*}
In summary, the KME of the IMQ-$C^{1/2}$ kernel is equal to the integrated $h$-depth (or the functional projection depth) \eqref{eq:integrated_depth} using as $\kappa_{1}$ the one-dimensional Gauss $h$-depth \eqref{eq:h_mode} and $N_{C}$ as the averaging measure $\nu$. The sample version of that integrated $h$-depth does not have to involve numerical approximation, and can be evaluated directly using \eqref{eq:sample KME} as
    \[  \ID(x;P_n) = n^{-1}\sum_{i=1}^{n} (\norm{C^{1/2}x-C^{1/2}X_i}_{\calX}^{2}+1)^{-1/2}.    \] 
\end{example}

\section{Asymptotics of Functional Depth Using KME}\label{sec:h_mode_KME}
 Recall that for a sample $X_1, \dots, X_n$ of i.i.d. random variables from $P \in \calP(\calX)$ we denote by $P_n \in \calP(\calX)$ the corresponding random empirical measure. Consistency results for $h$-depth obtained in \citep{Nagy2015} show almost complete convergence of $D_{\kappa}(x;P_{n})$ to $D_{\kappa}(x;P)$ uniformly across bounded subsets of the input space $\calX$. A different uniform consistency result for $h$-depth without rates of convergence can be found in \cite{Gijbels2017}. For (random) functional projection depth \eqref{eq:integrated_depth} no explicit results regarding the sample version asymptotics are available in the literature.

We are now ready to use the connections observed in Section~\ref{sec:Equivalance} to present an array of new theoretical results regarding the uniform consistency and asymptotic normality of the $h$-depth \eqref{eq:h_mode} and the functional projection depth \eqref{eq:integrated_depth}. We demonstrate that asymptotic results with far weaker assumptions on $\calX$ and with far less complicated proofs are possible to be derived by using the representation \eqref{eq:IPM}, and exploiting the observed connections between $h$-mode depth \eqref{eq:h_mode} and functional projection depth \eqref{eq:integrated_depth} with KMEs. We use the fact that the MMD in \eqref{eq:IPM} between $P_{n}$ and $P$ is an empirical process \citep[Chapter 3]{Gine2015}, for which concentration inequalities are readily available in the literature. 

\smallskip\noindent\textbf{Uniform consistency: Completely observed data.} Our first result involves a notion of convergence called \emph{almost complete convergence}. For a sequence of real-valued random variables $\{X_{n}\}_{n=1}^{\infty}$, a real-valued random variable $X$, and a sequence $\{a_{n}\}_{n=1}^{\infty}\subset[0,\infty)$ converging to zero we write $X_{n}-X = O_{a.co.}(a_{n})$ if there exists some $\varepsilon_{0}>0$ such that $\sum_{n=1}^{\infty}\bbP(\lvert X_{n}-X\rvert > a_{n}\varepsilon_{0})<\infty$. This notion of convergence is stronger than almost sure convergence \citep{Nagy2015}.

\begin{theorem}\label{thm:consistency}
Let $\calX$ be a separable topological space, $k$ be a bounded, continuous kernel on $\calX$ and $P\in\calP(\calX)$. Then for $P_{n}$ the empirical measure formed by i.i.d. observations from $P$ we have
\begin{align*}
    \sup_{x\in\calX}\left\lvert\Phi_{k}P_{n}(x)-\Phi_{k}P(x)\right\rvert = O_{a.co.}\left(\sqrt{\log n/n}\right).
\end{align*}
\end{theorem}
\begin{proof}
First note that by the reproducing property \eqref{eq:reproducing property} of the RKHS and the Cauchy-Schwarz inequality \cite[Corollary~5.1.4]{Dudley2002}
\begin{align*}
    \sup_{x\in\calX}\left\lvert\Phi_{k}P_{n}(x)-\Phi_{k}P(x)\right\rvert & = \sup_{x\in\calX}\left\lvert\langle k(\cdot,x),\Phi_{k}P_{n}-\Phi_{k}P\rangle_{k}\right\rvert\\
    &\leq \sup_{x\in\calX}\sqrt{k(x,x)}\norm{\Phi_{k}P_{n}-\Phi_{k}P}_{k}\leq c\,\norm{\Phi_{k}P_{n}-\Phi_{k}P}_{k},
\end{align*}
for some $c<\infty$ by the boundedness assumption on the kernel. Without loss of generality suppose $c = 1$ and for ease of notation set $E_{n} = \norm{\Phi_{k}P_{n}-\Phi_{k}P}_{k}$. By \eqref{eq:IPM} we know that $E_{n}$ is an empirical process indexed by the RKHS. As $k$ is continuous and $\calX$ is separable the RKHS is also separable \citep[Lemma 4.33]{Steinwart2008}, and we may apply simple concentration inequalities for empirical processes over separable Hilbert spaces. For example, \citep[Proposition A.1]{Tolstikhin2017} gives $\bbP(E_{n}>\varepsilon) \leq e^{-\frac{1}{2}(\sqrt{n}\varepsilon-1)^{2}}$. Setting $\varepsilon_{0} = 2$ in the definition of almost complete convergence gives
\begin{align*}
    \sum_{n=1}^{\infty}\bbP(E_{n}>\varepsilon_{0}\sqrt{\log n/n})\leq \sum_{n=1}^{\infty}e^{-\frac{1}{2}(4\log n -4\sqrt{\log n}+1)} = e^{-\frac{1}{2}}\sum_{n=1}^{\infty}n^{-2}e^{2\sqrt{\log n}} < \infty.
\end{align*}
\end{proof}

Theorem \ref{thm:consistency} is similar to \citep[Theorem 2]{Nagy2015} but the set over which the supremum is taken here does not depend on $n$ and the assumptions on $\calX$ are less restrictive.  By choosing a kernel which satisfies Theorems \ref{thm:h_mode_equivalence} and \ref{thm:consistency} we get a consistency result for $h$-depth. An analogous result for functional projection depth follows directly from Theorem~\ref{thm:integrated_depth} and Example~\ref{exp:integrated_depth}.

Having obtained the consistency and the rate of convergence of $h$-depth we may ask the same questions about the maximizer of the $h$-depth of $P$
	\[	\theta(P) = \argmax_{x \in \calX} D_{\kappa}(x;P).	\]
Exploiting the analogy of $h$-depth and functional pseudo-density estimation \citep{Ferraty_Vieu2006} the point $\theta(P)$ may be interpreted as a functional $h$-mode \cite{Cuevas2007} of the probability measure $P$. In case when the argument of maxima is not a single point set, any measurable selection from that set can be used in the definition of $\theta(P)$. A natural estimator of the functional $h$-mode is its empirical counterpart $\theta_n(P) = \argmax_{x \in \calX} D_{\kappa}(x;P_n)$. Under the assumption of uniqueness of $\theta(P)$, the statistic $\theta_n(P)$ estimates $\theta(P)$ consistently, with an explicit rate of convergence. The proof of this result follows directly from \cite[Section~9.7.3]{Ferraty_Vieu2006} and \cite[Theorem~1]{Dabo_etal2007}. An analogous result for functional projection depth \eqref{eq:integrated_depth} is immediate.

\begin{theorem}\label{thm:mode}
For a kernel satisfying Theorem~\ref{thm:h_mode_equivalence} and under the assumptions of Theorem~\ref{thm:consistency}, suppose in addition that for 
	\[	\calA_\varepsilon = \left\{ x \in \calX \colon D_{\kappa}(x;P) \geq D_{\kappa}(\theta(P);P) - \varepsilon \right\}.	\]
we can write $T(\varepsilon) = \sup_{x,y \in \calA_\varepsilon} \left\Vert x - y \right\Vert_\calX \to 0$ as $\varepsilon \to 0$. Then
\begin{align*}
    \left\Vert \theta_n(P) - \theta(P) \right\Vert_\calX = O_{a.co.}\left(T\left(\sqrt{\log n/n}\right)\right).
\end{align*}
\end{theorem}

A remarkable feature of $h$-depth and functional projection depth is that the convergence in Theorems~\ref{thm:consistency} and~\ref{thm:mode} holds over the whole infinite-dimensional space $\calX$, with no restrictions imposed on $P$. Furthermore, the rates of convergence of the depth do not implicitly depend on the structure, or the complexity of $\calX$, as well as they are not affected by the concentration properties of $P$. Results of this generality and simplicity are rare in nonparametric FDA  \cite{Ferraty_Vieu2006, Ferraty_etal2012}. The rate of convergence of the functional $h$-mode $\theta_n(P)$ does, of course, hinge on the degree of ``peakedness" of the $h$-depth measured by the diameter of its upper level set $\mathcal A_\varepsilon$. In specific cases, the latter quantity can be expressed explicitly, giving us an impression of the typical rates of convergence of the functional $h$-mode estimator.

\begin{example}\label{exp:mode_rates}
For a Gaussian measure $P = N_{m,C} \in \calP(\calX)$ in a separable Hilbert space $\calX$ with mean $m\in\calX$ and covariance operator $C \in L_1^+(\calX)$ the $h$-depth with the SE-$I$ kernel takes the form \citep{Kellner2019}
    \[  D_{\kappa}(x;P) = \left(\det\left(I + C \right)\right)^{-1/2} e^{-\frac{1}{2} \left\langle (I+C)^{-1}(x-m),x-m\right\rangle}.    \]
Here, $\det(I + C) = \prod_{j=1}^\infty(1 + \lambda_j)$ is the Fredholm determinant of the operator $I + C$ with eigenvalues $\left\{1 + \lambda_j \right\}_{j=1}^\infty$, which is well-defined since $C$ is trace class. The $h$-mode $\theta(P)$ is therefore the mean $m$, its depth is $D_{\kappa}(m;P) = \left(\det\left(I+C\right)\right)^{-1/2}$, and the set $\mathcal A_\varepsilon$ can be written as
    \[  \mathcal A_\varepsilon = \left\{ x \in \calX \colon \left\langle (I + C)^{-1}(x-m), x-m \right\rangle \leq c_\varepsilon \right\} \]
with $c_\varepsilon = - 2 \log \left( 1 - \varepsilon \left(\det\left( I + C \right)\right)^{1/2} \right)$. The level set $\mathcal A_\varepsilon$ is therefore the $\calX$-valued ``ellipsoid" of diameter $T(\varepsilon) = 2\sqrt{c_\varepsilon ( 1 + \lambda_1)}$ with $1 + \lambda_1$ the largest eigenvalue of the operator $I + C$. Since $c_\varepsilon = O(\varepsilon)$ we obtain that $T(\varepsilon) = O(\varepsilon^{1/2})$ and using Theorem~\ref{thm:mode} we see that the rate of convergence of the sample $h$-mode is $\left\Vert \theta_n(P) - \theta(P) \right\Vert_\calX = O_{a.co.}\left(\left(\log n/n\right)^{1/4}\right)$.
\end{example}

The following result is similar to \citep[Theorem 1]{Gijbels2015} where it is shown that the $J$-th order adjusted band depth \citep{Gijbels2015} is $\calP$-uniformly strongly consistent when considering a supremum over a uniformly equicontinuous set of functions on $[0,1]$. The adjusted band depth is a functional depth related to $h$-depth if one extends the definition of the former to include $J=1$ \citep[Section 3.2.2]{Nagy2016}. The following theorem improves on the latter consistency result for $h$-depth by not requiring the supremum to be taken over an equicontinuous subset, with more general assumptions on $\calX$. 

\begin{theorem}\label{thm:P_uniform_consistent}
Let $\calX$ be a separable topological space and $k$ be a bounded, continuous kernel on $\calX$. Then for every $\varepsilon > 0$
\begin{align*}
    \sup_{P\in\calP(\calX)}\bbP\left(\sup_{m\geq n} \sup_{x\in\calX}\left\lvert\Phi_{k}P_{m}(x)-\Phi_{k}P(x)\right\rvert > \varepsilon\right)\rightarrow 0.
\end{align*}
\end{theorem}
\begin{proof}
Copying the start of the proof of Theorem \ref{thm:consistency}, with the same notation, we have 
\begin{align}
    \bbP(\sup_{m\geq n}E_{m} >\varepsilon)\leq \sum_{m\geq n}\bbP(E_{m} > \varepsilon) \leq \sum_{m\geq n}e^{-\frac{1}{2}(\sqrt{m}\varepsilon - 1)^{2}}.\label{eq:tail_sum}
\end{align} 
The last term in \eqref{eq:tail_sum} converges to zero as $n\rightarrow\infty$ uniformly for all $P\in\calP(\calX)$ since the concentration inequality has no constant dependent on $P$. This completes the proof. 
\end{proof}

\smallskip\noindent\textbf{Uniform consistency: Discretely observed data.} In practice, functional data are seldom observed completely, simply because each random function typically attains a continuum of values. Instead, we standardly observe only discretised versions of the functional data, having the value of each functional observation $X_i \in \calX$ from $P_n \in \calP(\calX)$ recorded only at a finite grid in the domain, with possible measurement errors contaminating the discretised observations. In that setting, it is standard to first perform data pre-processing, consisting of approximating the unobserved function $X_i$ by its estimate $\widetilde{X}_i$ based on the available data. We will use $\widetilde{P}_{n} \in \calP(\calX)$ to denote an empirical distribution of the reconstructed functions $\widetilde{X}_1, \dots, \widetilde{X}_n$. Given assumptions on the quality of the reconstruction, consistency for the estimation of the KME is obtained. These results are, of course, directly applicable to $h$-depth and functional projection depth. 

The next theorem is similar to \citep[Theorem 7]{Nagy2019} but the proof leverages the KME representation of $h$-depth and uses the relationship between MMD and weak convergence to easily obtain the desired convergence result.

\begin{theorem}\label{thm:uniform_conv}
Let $\calX$ be a separable metric space, $k$ a bounded, continuous kernel on $\calX$ and $P\in\calP(\calX)$. Let $P_{n}$ be an empirical measure formed by i.i.d.\ observations from $P$ and let $\widetilde{P}_{n}$ be an approximation of $P_{n}$. Then $\bbP(\widetilde{P}_{n}\xrightarrow[]{w}P) = 1$ implies
\begin{align*}
    \sup_{x\in\calX}\left\lvert\Phi_{k}\widetilde{P}_{n}(x)-\Phi_{k}P(x)\right\rvert\xrightarrow[]{a.s.}0.
\end{align*}
\end{theorem}
\begin{proof}
As in the proof of Theorem \ref{thm:consistency} and Theorem \ref{thm:P_uniform_consistent} bound the supremum norm by $\text{MMD}_{k}(\widetilde{P}_{n},P)$. By the assumption on $\widetilde{P}_{n}$ we have $\text{MMD}_{k}(\widetilde{P}_{n},P)\rightarrow 0$ with probability one since MMD is dominated by weak convergence, see \citep[Lemma 10]{Simon-Gabriel2018} and \citep[Theorem 1.D.1]{Guilbart1978}.
\end{proof}

A simple sufficient condition ensuring validity of Theorem~\ref{thm:uniform_conv} is given in the next lemma. 

\begin{lemma}\label{lem:weak_conv}
Let $\calX$ be a separable metric space, $X_1, \dots, X_{n}, \dots$ a sequence of i.i.d. observations from $P\in\calP(\calX)$, and let $\widetilde{X}_{1}, \dots, \widetilde{X}_n, \ldots \in\calX$ be a sequence of its independent approximations. Then $\bbE[\norm{\tilde{X}_{n}-X_{n}}_{\calX}]\rightarrow 0$ as $n\rightarrow\infty$ implies $\bbP(\widetilde{P}_{n}\xrightarrow{w}P) = 1$.
\end{lemma}
\begin{proof}
As discussed in the proof of \citep[Theorem 1]{Nagy2019} the assumption on $\widetilde{X}_{n}$ and Markov's inequality show that $\widetilde{X}_{n}$ converges in law to $P$. Applying \citep[Lemma 2]{Nagy2016discrete} completes the proof. 
\end{proof}

Two common observation scenarios in FDA are those of densely, and sparsely observed functional data. In the former case, it is assumed that as the sampling process continues, each $X_n \in \calX$ is observed at an increasingly denser grid of time points in its domain. Under that assumption, Lemma~\ref{lem:weak_conv} is valid in, for example, the reconstruction scheme based on nonparametric smoothing as described in \citep{Nagy2019}. Note, however, that Theorem~\ref{thm:uniform_conv} applies also to situations when the random functions are observed sparsely, meaning that each $X_i \in \calX$ is known only at a finite grid of a limited number of time points. In that situation, information from the pooled sample of all the observed data must be utilised to obtain the reconstructed curves $\widetilde{X}_1, \dots, \widetilde{X}_n$, which are therefore typically no longer independent \cite{Yao_etal2005}. Our only requirement for validity of Theorem~\ref{thm:uniform_conv} is the Varadarajan-type assumption \cite[Theorem~11.4.1]{Dudley2002} of the almost sure weak convergence of $\widetilde{P}_n$ to $P$ that needs to be verified on a case-by-case basis, depending on the exact observation scenario, and the reconstruction method selected.

\smallskip\noindent\textbf{Asymptotic normality.} Aside from ascertaining that a KME (or a functional depth) can be estimated consistently, another desire is to know its asymptotic distribution. Since $\Phi_{k}P_{n}$ is an empirical mean in the RKHS $\calH_{k}$ we can simply employ Hilbert space CLT-type results. In the following theorem we establish a uniform central limit theorem for the sample KME, applicable to both $h$-depth and functional projection depth. Remarkably, the result holds true without any assumptions on the distribution $P \in \calP(\calX)$, over the whole --- possibly infinite-dimensional --- sample space $\calX$. It is the first result of its kind for a functional depth available in the literature. The only comparable results are \citep[Theorem~4]{Cuevas2009} and \citep[Theorem~1]{Nagy_etal2021} where much weaker, non-uniform versions of CLTs are derived for specific integrated functional depths, under restrictive conditions on the distribution $P \in \calP(\calX)$.

\begin{theorem}
Let $\calX$ be a separable topological space, $k$ be a bounded, continuous kernel on $\calX$ and $P\in\calP(\calX)$. Then
\begin{align*}
    n^{\frac{1}{2}}(\Phi_{k}P_{n}-\Phi_{k}P)\xrightarrow[]{d}Z,
\end{align*}
where $Z$ is a Gaussian random element in $\calH_{k}$ with covariance operator $C\in L^{+}_{1}(\calH_{k})$ given by $\langle Cf,g \rangle_{k} = \bbE_{X\sim P}[f(X)g(X)]$. 
\end{theorem}
\begin{proof}
By \citep[Lemma 4.33]{Steinwart2008} the RKHS $\calH_{k}$ is separable. Since $k$ is bounded, $\bbE_{X\sim P}[\norm{k(X,\cdot)}_{k}^{2}] = \bbE_{X\sim P}[k(X,X)]$ is finite. Therefore, we may apply a classic CLT in separable Hilbert spaces \citep[Theorem 2.1]{horvath2012inference}.
\end{proof}

For certain kernels a similar result holds even if the samples from $P$ are weakly dependent, namely if they are $L^{2}$-$m$-approximable in the sense of \citep{Hrmann2010}, see also \citep[Chapter 16]{horvath2012inference}. This is because the weak dependence is inherited by the kernel to become weak dependence in the RKHS, which allows \citep[Theorem 16.3]{horvath2012inference} to be applied.

\section{Characterising Probability Measures Using Functional Depth}\label{sec:characterising}
The depth function $\calX \to [0,\infty) \colon x \mapsto D(x;P)$ is a concept designed to serve as a representative of the underlying probability distribution $P \in \calP(\calX)$ in nonparametric statistics. As such, an important desideratum is the \emph{characterisation property} of a statistical depth, meaning that for any two different distributions $P \ne Q$ in $\calP(\calX)$, there should exist a point $x \in \calX$ that distinguishes $P$ and $Q$ in terms of their depth, i.e. $D(x;P) \ne D(x;Q)$. Despite the vast amount of research on the properties of various depth functions, not many depths are known to satisfy the characterisation property, even in the base case $\calX = \bbR^d$ \cite[Section~3.3]{Mosler_Mozharovskyi2021}. The derivation of characterisation results is generally not easy. Mainly for these reasons, the same problem in function spaces $\calX$ has not received much attention in the literature \citep[Section 8.2.4]{Nagy2016}. So far, no functional depth has been identified to satisfy the characterisation property.

Using Theorem \ref{thm:h_mode_equivalence} we see that our characterisation problem for $h$-depth and functional projection depth reduces to the question of when, for a kernel $k(x,y) = \kappa(\norm{x-y}_{\calX})$ with $\kappa$ satisfying Definition \ref{def:h_mode}, the mean embedding $\Phi_{k}$ is injective. Or, in other words, when does $\MMD_{k}(P,Q) = 0$ imply $P=Q$? Kernels for which this holds are called \emph{characteristic}. Not all kernels have this property.

\begin{example}
Suppose $P,Q\in\calP(\bbR^{d})$ have well defined means. Consider the kernel $k(x,y) = \langle x,y\rangle_{\bbR^{d}}$ on $\bbR^{d}$. Since $\int_{\bbR^{d}}\int_{\bbR^{d}}k(x,y)dP(x)dQ(y) = \langle \bbE_{X\sim P}[X],\bbE_{Y\sim Q}[Y]\rangle_{\bbR^{d}}$, by \eqref{eq:double_integral} we obtain that $\eMMD_{k}(P,Q) = \norm{\bbE_{X\sim P}[X]-\bbE_{Y\sim Q}[Y]}_{\bbR^{d}}$. Hence, $\eMMD_{k}(P,Q)=0$ if and only if $P$ and $Q$ have the same mean, and the kernel $k$ is not characteristic. 
\end{example}

The problem of identifying characteristic kernels has been studied widely in statistical machine learning \citep{Gretton2007,Sriperumbudur2010,Simon-Gabriel2018}. Most research has focused on the case where $\calX\subseteq\bbR^{d}$ for which many characteristic kernels are known \citep{Sriperumbudur2010}. Examples include the SE, IMQ and Mat\'ern kernels, directly providing examples of $h$-depths in $\bbR^d$ that characterise all probability distributions. A concept typically used in conjunction with characteristicness of a kernel is universality. Nevertheless, to realise the link the space $\calX$ needs to be locally compact \citep{Simon-Gabriel2018} which does not hold if $\calX$ is an infinite-dimensional Hilbert space, a common assumption in FDA. Therefore this link is hard to use in our main scenario of interest. 

Theory for characteristic kernels over function spaces $\calX$ has not been studied much. In \citep{Chevyrev2018} a characteristic kernel was derived for probability measures over the space of continuous functions of bounded variation using the signature transform \citep{chevyrev2016primer} which involves an infinite sum in the kernel. In practice, however, the sum has to be truncated, meaning the resulting kernel used in practice is not characteristic. In \citep{Hayati2020} it was shown that the SE-$I$ kernel can distinguish Gaussian probability measures using explicit formulas of the KMEs. A natural generalisation of the established theory of characteristic kernels was presented in \citep{Wynne2020a} which dealt with kernels on separable Hilbert spaces. The main idea is to use the spectral properties of the kernel to prove characteristicness. The next result was proven for $\calX = \bbR^{d}$ in \citep{Sriperumbudur2010} and the proof for an arbitrary Hilbert space is analogous.

\begin{theorem}\label{thm:spectral}
Let $\calX$ be a Hilbert space and $k(x,y) = \widehat{\mu}(x-y)$ for some $\mu\in\calP(\calX)$. Then $\mu$ having full support implies that $k$ is characteristic.
\end{theorem}
\begin{proof}
By \eqref{eq:L2} if $\text{MMD}_{k}(P,Q) = 0$ and $\mu$ has full support then $\widehat{P} = \widehat{Q}$ on every open set of $\calX$. Since $\widehat{P},\widehat{Q}$ are continuous this gives that $\widehat{P} = \widehat{Q}$ and so, by injectivity of the Fourier transform of probability measures \cite{Gine2015}, we obtain $P=Q$.
\end{proof}

The next two examples use Theorem \ref{thm:spectral} to give characteristic kernels. They involve the condition of $C\in L^{+}_{1}(\calX)$ being injective. This is equivalent to $\langle Cx,x\rangle_{\calX} = 0$ if and only if $x = 0$ and holds, for example, if all the eigenvalues of $C$ are strictly positive. In the particular case of $\calX = L^{2}(\calD)$ for $\calD\subseteq{\bbR^{d}}$ it is known that every such operator $C$ is of the form $Cx(s) = \int_{\calD}k_{C}(s,t)x(t)dt$ for some kernel $k_{C}$ on $\calD$ \citep[Theorem 2.1]{Fasshauer2015}. Therefore, if $k_{C}$ is integrally strictly positive definite, meaning $\int_{\calD}\int_{\calD}f(s)k_{C}(s,t)f(t)dsdt \geq 0$ with equality if and only if $f = 0$, then $C$ is injective. For a list of integrally strictly positive kernels see \citep{Sriperumbudur2010}.

\begin{example}\label{exp:SE}
For any $C\in L^{+}_{1}(\calX)$ the SE-$C^{1/2}$ kernel is the Fourier transform of the centred Gaussian measure $N_{C}$ with covariance operator $C$ on $\calX$, see \eqref{eq:SE relation} in Section \ref{subsec:examples}. The measure $N_C$ has full support if $C$ is injective \citep[Proposition 1.2.5]{DaPrato2006}, therefore the SE-$C^{1/2}$ kernel is characteristic if $C$ is injective.  
\end{example}

\begin{example}
In the notation of Theorem \ref{thm:integrated_depth} suppose that $U \sim \mu_{1} \in \calP(\bbR)$ and $V \sim \nu \in \calP(\calX)$ are such that the distribution $\mu \in \calP(\calX)$ of $U V$ has full support in $\calX$. Then the kernel corresponding to the integrated depth \eqref{eq:integrated_depth} is characteristic by Theorem \ref{thm:spectral}. This applies to the IMQ-$C^{1/2}$ kernel, used in Example \ref{exp:integrated_depth}, if $C$ is injective since $\mu_{1}$ being the standard normal has full support in $\bbR$ and $\nu = N_{C}$ has full support too. In particular, the corresponding functional projection depth \eqref{eq:integrated_depth} characterises all probability distributions in $\calX$.
\end{example}

Example \ref{exp:SE} is a simple way to show that the SE kernels are characteristic for certain choices of $C$. The assumption on $C$ is, however, still limiting. For example, it is not satisfied by the identity operator employed in the classical Gauss $h$-depth. The same line of argument was nevertheless extended using a limiting argument in \citep[Theorem 4]{Wynne2020a} to obtain the next result. A consequence of this result is that, for instance, the Gauss $h$-depth also completely characterises all probability measures in separable Hilbert spaces.

\begin{theorem}
Let $\calX,\calY$ be real, separable, Hilbert spaces and $T\colon\calX\rightarrow\calY$ be Borel measurable, continuous and injective. Then SE-$T$ and IMQ-$T$ kernels are characteristic. 
\end{theorem}

Drawing from the advances in statistical kernel methods, we have demonstrated that multiple kernels commonly used in Hilbert spaces are characteristic. This means that the corresponding $h$-depths and functional projection depths under mild conditions characterise all distributions from $\calP(\calX)$. We obtain the first examples of statistical depths that completely satisfy the characterisation property not only in the Euclidean space $\bbR^d$, but also in infinite-dimensional Hilbert function spaces $\calH$ such as the space $L^2([0,1])$. Consequently, in statistical inference of $P \in \calP(\calX)$ based exclusively on $h$-depth or integrated depths \eqref{eq:integrated_depth}, no information about $P$ is lost, and these depths can be used as tantamount to the probability distribution itself. This observation opens novel research directions in nonparametric FDA, and the related depth-based statistics.


\section{MMD and Empirical Characteristic Function Based Testing}\label{sec:testing}
The main use of MMD in practice, and its original purpose in statistical machine learning, is to perform non-parametric tests \citep{Gretton05,Gretton2007}. We begin by considering the two-sample problem \citep{Gretton2012}, where for $P,Q\in\calP(\calX)$ we are given two independent samples $\{X_{i}\}_{i=1}^{n},\{Y_{i}\}_{i=1}^{n}$ of  i.i.d.\ observations from $P$ and $Q$, respectively. Our intention is to test the null hypothesis $H_{0}\colon P=Q$ against the alternative $H_{1}\colon P\neq Q$. We are assuming for simplicity that we observe an equal number of observations from each distribution but this can be easily generalised.

A natural test statistic based on the expression for MMD from \eqref{eq:double_integral} is the U-statistic
\begin{align}\label{eq:unbiased_U}
    \widehat{\MMD}_{k}(P,Q)^{2} = \sum_{1\leq i < j\leq n} \left(k(X_{i},X_{j})+k(Y_{i},Y_{j})-k(X_{i},Y_{j})-k(X_{j},Y_{i})\right).
\end{align}
The idea is that for $P=Q$ should this test statistic, which approximates $\MMD_{k}(P,Q)^{2}$, be small. The asymptotic distribution of \eqref{eq:unbiased_U} under the null hypothesis is derived from the standard $U$-statistics theory \citep{Gretton2012}. The test procedure is performed using permutation bootstrap, for more details see \citep{Gretton2012}. The two-sample test based on \eqref{eq:unbiased_U} has had wide success in a range of applications given the different types of setups kernels can be defined in, such as random vectors, graphs, or functional data \citep{Gretton2012,Borgwardt2006,Wynne2020a}. 

The connection between MMD-based tests and the existing tests in FDA have so far gone unnoticed, likely due to the way MMD has only recently been applied to functional data. The key link is that if $k(x,y) = \widehat{\mu}(x-y)$ for some $\mu \in \calP(\calX)$ then the representation \eqref{eq:L2} shows that MMD, and hence the MMD-based tests, can be interpreted in terms of a distance between characteristic functions. Therefore, tests based on the empirical approximations of characteristic functions, such as those presented in \citep{Jiang2019} and \citep{Henze2020}, are in fact MMD-based tests. We will discuss two examples, one for two-sample testing for functional data in Section~\ref{sec:two sample} and one for testing Gaussianity of random functions in Section~\ref{sec:Gaussianity}.

\subsection{Connection to Empirical Characteristic Function Two-Sample Test}    \label{sec:two sample}

A recent two-sample test based on empirical characteristic functions of discretised random functions \citep[Equation (15)]{Jiang2019} is very similar to a biased estimate of MMD when using an SE kernel \eqref{eq:SE kernel}. In \citep{Jiang2019} the authors consider functions which may have multi-dimensional outputs. We will consider only univariate outputs for brevity but the conclusions in the general case are the same.

Let $P,Q\in\calP(L^{2}([0,1]))$ and let $\{X_{i}\}_{i=1}^{n}\overset{i.i.d.}{\sim}P$ be independent with $\{Y_{i}\}_{i=1}^{n}\overset{i.i.d.}{\sim}Q$. Suppose we observe $m$-dimensional discretisations of these samples in equidistant points in $[0,1]$ denoted by $\{X_{i}^{(m)}\}_{i=1}^{n}$, $\{Y_{i}^{(m)}\}_{i=1}^{n}$. For example we may suppose that the function $X_i$ is represented by $X_{i}^{(m)} = (X_{i}(0),X_{i}(1/m),\ldots,X_{i}((m-1)/m))$, and analogously for $Y_i$. Denote by $P_{n}^{(m)},Q_{n}^{(m)} \in \calP(\bbR^m)$ the empirical distributions, each formed from $n$ functions observed at $m$ locations. The test in \citep{Jiang2019} is based upon the weighted distance between the empirical characteristic functions of $P_{n}^{(m)}$ and $Q_{n}^{(m)}$ \begin{align*}
    T_{n} = \int_{\bbR^{m}}\left\lvert \widehat{P_{n}^{(m)}}(v)-\widehat{Q_{n}^{(m)}}(v)\right\rvert^{2}e^{-\frac{\norm{v}_{\bbR^{m}}^{2}}{2}}dv.
\end{align*}
Following a similar derivation to the proof of Theorem \ref{thm:MMD_reps} one can obtain \citep[Equation (15)]{Jiang2019}
\begin{align}\label{eq:biased_U}
    T_{n} = \frac{1}{n^{2}}\sum_{i,j=1}^{n}e^{-\norm{X_{i}^{(m)}-X_{j}^{(m)}}_{\bbR^{m}}^{2}/2} - \frac{2}{n^{2}}\sum_{i,j=1}^{n}e^{-\norm{X_{i}^{(m)}-Y_{j}^{(m)}}_{\bbR^{m}}^{2}/2} +\frac{1}{n^{2}}\sum_{i,j=1}^{n}e^{-\norm{Y_{i}^{(m)}-Y_{j}^{(m)}}_{\bbR^{m}}^{2}/2}.
\end{align}
We see that \eqref{eq:biased_U} is almost the same as \eqref{eq:unbiased_U} using $\calX = \bbR^{m}$ and $k(x,y) = e^{-\norm{x-y}_{\bbR^{m}}^2/2}$. The only exception is that \eqref{eq:biased_U} is a biased estimator of the corresponding MMD due to including $i=j$ in the sum in the negative term. In summary, the test statistic used in \citep{Jiang2019} is the same as a biased estimator of MMD, using an SE kernel, between discretised versions of the random functions. This discretisation causes the test to not be able to truly distinguish between $P,Q$ since it only checks the joint distribution at the observation locations. A solution to this problem has been proposed in the literature. It can be shown that if approximations of a certain level of quality are made based on the discretisations then the test statistic will behave as if we had non-discretised samples, see \citep[Theorem 5]{Wynne2020a}.

\subsection{Connection to Empirical Characteristic Function Gaussianity Test}   \label{sec:Gaussianity}

A test for Gaussianity of functional data performed in \citep{Henze2020} uses empirical characteristic functions. That may be viewed as an MMD-based test of normality studied in \citep{Kellner2019}. Let $\calX$ be a Hilbert space, we observe $\{X_{i}\}_{i=1}^{n}\overset{i.i.d.}{\sim}P$ for some $P\in\calP(\calX)$ and we want to test whether $P$ is Gaussian. The following test statistic was proposed in \citep[Equation (4)]{Henze2020} 
\begin{align}
    T_{n} = \int_{\calX}\left\lvert \widehat{P_{n}}(v)-\widehat{N}_{m_{n},\Sigma_{n}}(v)\right\rvert^{2}dN_{C}(v).\label{eq:henze_stat}
\end{align}
Here $m_{n},\Sigma_{n}$ are, respectively, the empirical mean and covariance operator of $\{X_{i}\}_{i=1}^{n}$, $N_{m_{n},\Sigma_{n}}$ is the Gaussian measure with mean $m_{n}$ and covariance operator $\Sigma_{n}$, and $N_{C}$ is some zero mean Gaussian measure on $\calX$ with $C \in L_1^+(\calX)$ given. 

Once again, following the proof of Theorem \ref{thm:MMD_reps} we can rewrite \eqref{eq:henze_stat} using MMD as
\begin{align*}
    T_{n} & = \frac{1}{n^{2}}\sum_{i=1}^{n}k(X_{i},X_{j})-\frac{2}{n}\sum_{i=1}^{n}\int_{\calX}k(X_{i},y)dN_{m_{n},\Sigma_{n}}(y) + \int_{\calX}\int_{\calX}k(x,y)dN_{m_{n},\Sigma_{n}}(x)dN_{m_{n},\Sigma_{n}}(y)\\
    & = \frac{1}{n^{2}}\sum_{i=1}^{n}k(X_{i},X_{j}) - \frac{2}{n}\sum_{i=1}^{n}\Phi_{k}N_{m_{n},\Sigma_{n}}(X_{i}) + \norm{\Phi_{k}N_{m_{n},\Sigma_{n}}}_{k}^{2},
\end{align*}
where $k(x,y) = \widehat{N_{C}}(x-y)$ is the SE-$C^{1/2}$ kernel \eqref{eq:SE kernel} from Section \ref{subsec:examples}. This is the same test statistic as the one derived in \citep[Equation (4.1)]{Kellner2019}. The terms involving $\Phi_{k}N_{m_{n},\Sigma_{n}}$ can be calculated in closed form using standard Gaussian measure integral identities, see \citep[Proposition 4.2]{Kellner2019} and \citep[Theorem 6]{Wynne2020a}. 

Using the observed link to MMD means that an empirical average to approximate the integral in \eqref{eq:henze_stat} is not required. Such an empirical approximation was used in the numerical experiments in \citep{Henze2020}, introducing an unnecessary extra computational cost. In summary, the test for Gaussianity of functional data based on the empirical characteristic function presented in \citep{Henze2020} coincides with the MMD-based tests for normality in \citep{Kellner2019} when using an SE-$C^{1/2}$ kernel.

\section{Conclusion}\label{sec:Conclusion}
In this paper we have shown that both commonly used $h$-depth and (random) functional projection depths for functional data may be realised as KMEs. As a result, alternative and shorter proofs of existing consistency results are presented with more general assumptions. Novel results regarding asymptotic normality of the sample functional depths are provided. The characterisation of probability measures via their $h$-depth and functional projection depth is proved by using characteristic kernels. Finally, tests based on empirical characteristic functions were shown to be equivalent to MMD-based tests due to the spectral representation \eqref{eq:L2} of MMD. 

The future work exploring additional interplays of machine learning and FDA that was not considered in this contribution would consist of employing advances in kernel-based testing to functional data. This includes \begin{enumerate*}[label=(\roman*)] \item independence testing using the \emph{Hilbert-Schmidt independence criterion} \citep{Gretton05} which is the MMD between a joint and a product measure; \item using linear computational cost estimators of MMD \citep{Gretton2012}; and \item goodness-of-fit tests \citep{Liu2016,Chwialkowski2016}.\end{enumerate*} Aside from testing, MMD is also successfully used as a distance for performing parameter inference \citep{Briol2019}, which could be of interest especially in FDA. Specific applications of MMD-based testing to FDA have been demonstrated in \citep{Wynne2020a,Hayati2020}, but there certainly remains a broad space for further applications. 

\section*{Acknowledgements}
The research of G. Wynne was supported by an EPSRC Industrial CASE award [EP/S513635/1] in partnership with Shell UK Ltd. The research of S. Nagy was supported by the grant 19-16097Y of the Czech Science Foundation, and by the PRIMUS/17/SCI/3 project of Charles University. We thank Andrew B. Duncan for helpful comments.

{
\bibliographystyle{abbrvnat}
\bibliography{inf_dim_refs.bib,depth.bib}

}



\end{document}